\documentclass[]{article}

\usepackage[leqno]{amsmath}
\usepackage{amssymb,amsfonts,xfrac,MnSymbol}
\usepackage{amsthm}
\usepackage{cite}
\usepackage{hyperref}
\usepackage{todonotes}
\usepackage{enumitem}
\usepackage{graphicx}
\usepackage{tikz-cd, tikz}
\usepackage{color}
\usepackage{import}

\numberwithin{equation}{section}

\newtheorem{theorem}{Theorem}[section]

\newtheorem{proposition}[theorem]{Proposition}
\newtheorem{lemma}[theorem]{Lemma}

\newtheorem{corollary}[theorem]{Corollary}

\newtheorem*{theorem*}{Theorem}
\newtheorem*{claim*}{Claim}
\newtheorem*{proposition*}{Proposition}
\newtheorem*{lemma*}{Lemma}
\newtheorem*{corollary*}{Corollary}

\newtheorem{theoremA}{Theorem}

\theoremstyle{definition}

\newtheorem{remark}[theorem]{Remark}
\newtheorem{example}[theorem]{Example}
\newtheorem{question}[theorem]{Question}

\newtheorem*{definition*}{Definition}
\newtheorem*{observation*}{Observation}
\newtheorem*{remark*}{Remark}
\newtheorem*{example*}{Example}
\newtheorem*{question*}{Question}
\newtheorem*{exercise*}{Exercise}
\newtheorem*{fact*}{Fact}
\newtheorem*{notation*}{Notation}

\newcommand{\bbH}{\mathbb{H}}

\newcommand{\bbN}{\mathbb{N}}

\newcommand{\bbR}{\mathbb{R}}

\newcommand{\bbZ}{\mathbb{Z}}

\newcommand{\bfG}{\mathbf{G}}

\newcommand{\calF}{\mathcal{F}}
\newcommand{\calG}{\mathcal{G}}
\newcommand{\calH}{\mathcal{H}}

\newcommand{\actson}{\curvearrowright}
\newcommand{\actsno}{\curvearrowleft}
\newcommand{\cechH}{\check{H}}

\newcommand{\into}{\hookrightarrow}

\newcommand{\ii}{^{-1}}

\newcommand{\gen}[1]{\left< #1 \right>}

\newcommand{\tild}[1]{\widetilde{#1}}

\DeclareMathOperator{\Stab}{Stab}

\DeclareMathOperator{\Aut}{Aut}

\DeclareMathOperator{\diam}{diam}

\DeclareMathOperator{\Comm}{Comm}
\DeclareMathOperator{\diff}{diff}

\DeclareMathOperator{\im}{Im}
\DeclareMathOperator{\Vol}{Vol}
\DeclareMathOperator{\rank}{rank}
\DeclareMathOperator{\Cx}{C}
\DeclareMathOperator{\covol}{Covol}
\DeclareMathOperator{\Isom}{Isom}

\DeclareMathOperator{\I}{I}
\DeclareMathOperator{\length}{length}

\title{Volume vs. Complexity of Hyperbolic Groups}
\author{Nir Lazarovich\thanks{Supported by the Israel Science Foundation (grant no. 1562/19), and by the German-Israeli Foundation for Scientific Research and Development.}}
\date{}

\begin{document}

\maketitle

\begin{abstract}
    We prove that for a one-ended hyperbolic graph $X$, the size of the quotient $X/G$ by a group $G$ acting freely and cocompactly bounds from below the number of simplices in an Eilenberg-MacLane space for $G$.
    We apply this theorem to show that one-ended hyperbolic cubulated groups (or more generally, one-ended hyperbolic groups with globally stable cylinders \`a la Rips-Sela) cannot contain isomorphic finite-index subgroups of different indices.
\end{abstract}

\section{Introduction}
\paragraph{Complexity vs volume.}
Let $G$ act freely and cocompactly on $\bbH^n$, $n\ge 2$. For $n=2$, the Gauss-Bonnet formula shows that the Euler characteristic of the group $G$ is proportional to the volume of the quotient $\bbH^2/G$.
In higher dimensions, Gromov and Thurston \cite{gromov1982volume,thurston1979geometry} showed that the simplicial volume of $G$ is proportional to the hyperbolic volume of the quotient $\bbH^n/G$.

In the above results, a topological feature of the group $G$ is shown to be related to the volume of the quotient $\bbH^n/G$. 
This phenomenon was extensively studied for groups acting on non-positively curved manifolds \cite{cooper1999volume,delzant2013complexity,ballmann1985manifolds,gelander2004homotopy,gelander2011volume,gelander2019minimal,bader2020homology,belolipetsky2010counting,gelander2021bounds}. 
In this work, we study the relation between topological complexity and volume in the discrete setting of group actions on hyperbolic simplicial complexes.
For this purpose let us first define what we mean by ``complexity'' and ``volume''. 

Let $G$ be a group. An Eilenberg-MacLane space $K=K(G,1)$ for $G$ is a CW-complex such that $\pi_1(K)=G$ and $\pi_n(K)=0$ for all $n\ge 2$.
We define the \emph{topological complexity} $\Cx(G)$ of $G$ to be the minimal number of cells in an Eilenberg-MacLane $\Delta$-complex for $G$.
The topological complexity of a group can be thought of as the group analogue of the Kneser complexity $k(M)$ of a manifold $M$ \cite{kneser1929geschlossene} which is the minimal number of simplices in a triangulation of $M$.

For a locally finite connected $\Delta$-complex $X$, we define the \emph{volume} $\Vol_X(G)$ of a free action $G\actson X$ to be the number of $G$-orbits of cells of $X$.
Note that if $X$ is contractible and $G$ acts on $X$ freely, then $X/G$ is an Eilenberg-MacLane space for $G$ and it follows in this case that $\Cx(G) \le \Vol_X(G)$.
Thus, we will be interested in bounding the complexity $\Cx(G)$ of $G$ from below by $\Vol_X(G)$.


Lower bounds do not hold in general. 
For example, if we consider actions of $\bbZ^n\actson \bbR^n$ by translations, then while the complexity of $\bbZ^n$ is fixed, the volume of the action can be arbitrarily large.
However, for one-ended hyperbolic graphs Theorems \ref{thm: main sqrt bound} and \ref{thm: main volume complexity} below provide lower bounds for the complexity in terms of the volume.

\begin{theoremA}\label{thm: main sqrt bound}
Let $X$ be a one-ended hyperbolic graph. Then there exists $\alpha=\alpha(X)>0$ such that if $G\actson X$ freely and cocompactly then $$\alpha \cdot \sqrt{\Vol_{X}(G)} \le \Cx(G).$$
\end{theoremA}

The next theorem provides a linear lower bound for one-ended hyperbolic spaces $X$ which admit $\Aut(X)$-globally stable cylinders \`a la Rips-Sela \cite{rips1995canonical} (see \S\ref{sec: global stability} for definitions). These include locally finite $C'(1/8)$-small cancellation polygonal complexes by Rips-Sela \cite{rips1995canonical} and hyperbolic CAT(0) cube complexes by L.-Sageev \cite{stabilityforCCC}. It is unknown whether all hyperbolic groups admit globally stable cylinders.

\begin{theoremA}\label{thm: main volume complexity}
    Let $X$ be a one-ended hyperbolic graph which admits $\Aut(X)$-globally stable cylinders. Then there exists $\alpha=\alpha(X)>0$ such that if a group $G$ acts freely and cocompactly on $X$ then $$\alpha\cdot \Vol_X(G) \le \Cx(G).$$
\end{theoremA}

The following is an immediate application of the inequalities above to the case of finite-index subgroups.

\begin{theoremA}\label{thm: main}
Let $G$ be a torsion-free one-ended hyperbolic group. Then, there exist $\alpha,\beta>0$ such that if $H\le G$ is a finite index subgroup then $$\alpha\cdot \sqrt{[G:H]} \le \Cx(H) \le \beta \cdot [G:H].$$

If moreover $G$ admits globally stable cylinders then $$\alpha\cdot [G:H] \le \Cx(H) \le \beta \cdot [G:H].$$
\end{theoremA}





    Note that $\Cx(G)<\infty$ if and only if $G$ is of type $F$. For hyperbolic groups this happens if and only if $G$ is torsion free. 
More refined versions of the above theorems for hyperbolic groups with torsion appear in \S\ref{sec: proof of thm A}.

The proofs of Theorems \ref{thm: main sqrt bound}, \ref{thm: main volume complexity}, and \ref{thm: main}  appear in \S\ref{sec: proof of thm A}. We sketch the main idea here.

\emph{Sketch of proof of Theorem \ref{thm: main volume complexity}.}
Let $K$ be a the $K(G,1)$ with the minimal number of cells. 
We may assume that $X$ is the 1-skeleton of a contractible finite dimensional complex (namely, its Rips complex).
The free action of $G\actson X$ endows the 2-skeleton of $K$ with a singular foliation by works of Delzant \cite{delzant1995image} and Rips-Sela \cite{rips1995canonical}. 
The details of this foliation are described in \S\ref{sec: global stability} and \S\ref{sec: the singular foliation}.
In Lemma \ref{lem: accessibility upper bound}, we provide an upper bound, that depends linearly on the number of 2-cells in $K$, for the number of leaves in this foliation.
On the other hand, the action $G \actson X$ induces a quasi-isometry between $G$ and $X$.
Using Bestvina-Mess \cite{bestvina1991boundary}, we show in Proposition \ref{prop: uniform quasi-onto} that any continuous quasi-isometry between Rips complexes of hyperbolic groups must be \emph{uniformly} quasi-surjective. 
Proposition \ref{prop: uniform quasi-onto} is used to get a lower bound, that depends linearly on $\Vol_{X}(H)$, on the number of leaves of the foliation.
Combining the two inequalities we get the desired result.

\paragraph{Isomorphic finite-index subgroups.} 
We now turn to an application of Theorem \ref{thm: main} to the following question. 
\begin{question}\label{q: hyperbolic groups have isomorphic finite index subgroups}
Let $G$ be a non-elementary hyperbolic group. Can $G$ contain two isomorphic finite index subgroups of different indices?
\end{question}

Note that this is possible for $G\simeq \bbZ$, so the assumption that $G$ is non-elementary is necessary for a negative answer. 
By Sela \cite{sela1997structure} torsion-free one-ended hyperbolic groups are co-Hopfian, so the answer is negative if one of the subgroups is equal to $G$. 
The answer is negative if the group has non-zero Euler characteristic, e.g free groups (of rank $\ge 2$) and (hyperbolic) surface groups.
By Mostow Rigidity \cite{mostow1968quasi}, it is also negative for hyperbolic manifold groups since the volume is an isomorphism invariant and depends linearly on the index.
By Sykiotis \cite{sykiotis2018complexity}, the answer is negative if $G$ has a JSJ decomposition with a maximal hanging Fuchsian group.
Surprisingly, Stark-Woodhouse \cite{stark2018hyperbolic} found examples of one-ended hyperbolic groups with isomorphic finite-index and infinite-index subgroups. 

We first answer Question \ref{q: hyperbolic groups have isomorphic finite index subgroups} in the case $G$ has more than two ends. 
\begin{theoremA}\label{thm: free splitting}
Let $G$ be a finitely presented group, then either $G$ has at most 2 ends or $G$ does not contain two isomorphic subgroups of different indices.
\end{theoremA}

If $G$ is virtually torsion free, Theorem \ref{thm: free splitting} is a corollary of \cite{sykiotis2018complexity,tsouvalas2018euler}. Its proof is independent of the rest of the paper and occupies \S\ref{sec: multiended}.

Finally we answer the question under the additional assumption that $G$ admits globally stable cylinders (e.g cubulated).
\begin{theoremA}\label{thm: isomorphic finite index}
Let $G$ be a non-elementary hyperbolic group that admits globally stable cylinders, then $G$ does not contain two isomorphic finite index subgroups of different indices.
\end{theoremA}

The proof of Theorem \ref{thm: isomorphic finite index} occupies \S\ref{sec: abstract comm}, and uses a bootstrapping argument, Lemma \ref{lem: arbitrary indices}, that shows that if $G$ contains two isomorphic finite index subgroups of different indices then one can assume that the ratio of their indices is arbitrarily large.





\paragraph{Related results and open questions.} \label{subsec: discussion}
We end this introduction with a discussion of related results and open questions. Theorem \ref{thm: main volume complexity} assumes that $G$ admits globally stable cylinders to get a linear lower bound.
\begin{question}[Rips-Sela]
    Do all hyperbolic groups admit globally stable cylinders?
\end{question}
\begin{question}
    Does the linear lower bound in Theorem \ref{thm: main volume complexity} hold for all non-elementary hyperbolic groups? 
\end{question}

Let $X$ be a locally finite $\Delta$-complex. If $G$ acts freely and cocompactly on $X$ then we can think of $G$ as a uniform lattice in $\Aut(X)$. Upon fixing a Haar measure on $\Aut(X)$, we can define the co-volume $\covol(G)=\Vol(\Aut(X)/G)$. It is easy to see that $\covol(G)\asymp_X \Vol_X(G)$ for every group $G$ acting freely and cocompactly on $X$. 
Thus, one can study the relation between measures of complexity of lattices and their co-volumes.
In the setting of non-positive curvature we ask:
\begin{question}
    Let $X$ be a geodesically complete CAT(0) space without Euclidean factors, does $\covol(G)\prec \Cx(G)$ hold for torsion-free uniform lattices $G$ in $\Isom(X)$?
\end{question}

There are other natural notions of complexity for groups. The most basic one is $\rank(G)$ -- the minimal number of generators. As the following example shows Theorem \ref{thm: main} and hence Theorems \ref{thm: main sqrt bound} and \ref{thm: main volume complexity} are false when replacing $\Cx(G)$ by $\rank(G)$.
\begin{example}
 Let $G$ be the fundamental group of a closed hyperbolic fibered 3-manifold group. That is, $G = S \rtimes \gen{\phi}$ where $S$ is a (closed) surface group, and $\phi\actson S$ is a pseudo-Anosov mapping class. Then the subgroups $G_n=S\rtimes \gen{\phi^n}\le G$ have $[G:G_n]=n$ but $\rank(G_n)\le \rank(S)+1$.
\end{example} 

The same example shows that the minimal number of relations in a presentation will also not be bounded from below. Perhaps the next natural notion of complexity is Delzant's T-invariant \cite{delzant1996decomposition} defined to be the minimal number $T(G)$ of relations of length 3 in a presentation for $G$ in which all relations have length $\le 3$.
Cooper \cite{cooper1999volume} showed that for a hyperbolic 3-manifold $M$ one has  $\Vol(M)\le \pi\cdot  T(\pi_1(M))$. 

\begin{question} 
    Can $\Cx(G)$ be replaced by $T(G)$ in Theorems \ref{thm: main sqrt bound}, \ref{thm: main volume complexity} and \ref{thm: main}? (cf. Theorems \ref{thm: main volume complexity upgraded} and \ref{thm: main upgraded}.)
\end{question}

As the following example shows, one should not expect the volume to bound $\Cx(G)$ or $T(G)$ from above in the non-discrete setting.
\begin{example}[Thurston \cite{thurston1979geometry}]\label{ex: Dehn fillings}
 Let $M$ be a cusped hyperbolic 3-manifold. By Dehn filling $M$, one can get non-isomoprhic hyperbolic 3-manifolds $M_i$ with $\Vol(M_i)<\Vol(M)$. Since $M_i$ are non-isomorphic, by Mostow Rigidity \cite{mostow1968quasi} they have non-isomorphic fundamental groups $G_i=\pi_1(M_i)$, which thus have arbitrarily large $\Cx(G_i)$ and $T(G_i)$.
\end{example}

With this exception in mind, Gelander \cite[Conjecture 1.3]{gelander2004homotopy} conjectures that if $G$ acts cocompactly and freely on a non-compact symmetric space $X$ (and $G$ is arithmetic if $\dim(X)=3$), then a minimal Eilenberg-Maclane simplicial complex for $G$ with bounded vertex degree (that depends only on $X$) has at most $\beta\cdot \Vol(X/G)$ vertices. The conjecture holds for non-uniform arithmetic lattices \cite{gelander2004homotopy}, and a slightly weaker bound was recently obtained in \cite{gelander2021bounds}.
Other measures of complexity that are bounded from above by the co-volume include: Betti numbers of non-positively curved manifolds \cite{ballmann1985manifolds}; $\rank(G)$ for lattices in various symmetric spaces \cite{belolipetsky2010counting, gelander2019minimal,gelander2011volume}; $\log(|Tor(H_k(G;\bbZ)|)$  for manifolds with normalized bounded negative curvature\cite{bader2020homology};  a relative T-invariant for hyperbolic 3-manifolds \cite{delzant2013complexity}.

Regarding Question \ref{q: hyperbolic groups have isomorphic finite index subgroups} and Theorems \ref{thm: free splitting} and \ref{thm: isomorphic finite index} we ask the following.
\begin{question}
Which groups have isomorphic finite index subgroups of different indices?
\end{question}
In particular, such groups include the class of \emph{finitely non-co-Hopfian groups} defined in \cite{bridson2010cofinitely}, i.e groups that are isomorphic to proper finite-index of themselves. Some results and conjectures highlight ``nilpotent features'' in finitely non-co-Hopfian groups (see \cite{van2017structure,nekrashevych2008scale}), can similar phenomena occur for groups with isomorphic finite index subgroups of different indices?

\paragraph{Acknowledgements.} The author would like to thank  Alex Margolis and Michah Sageev for their helpful suggestions and improvements of the results and the exposition, and Uri Bader, Mladen Bestvina, Tsachik Gelander, Yair Glasner, Zlil Sela and Daniel Woodhouse for valuable conversations and comments on the manuscript. 

\section{Cylinders, slices and global stability}\label{sec: global stability}
In this section we recall (and slightly adapt) the definition and properties of cylinders, slices and global stability from Rips-Sela \cite{rips1995canonical}.

Let $X$ be a graph. Let $d$ be the minimal path metric on the vertices of $X$.
Assume $X$ is $\delta$-hyperbolic.

\paragraph{Cylinders and slices.}
A $\theta$-\emph{cylinder} from $x$ to $y$ is the triple $(C,x,y)$ where $x,y\in X$ and $C$ is a subset that satisfies $\gamma \subseteq C\subseteq N_{\theta}(\gamma)$ for every geodesic $\gamma$ connecting $x,y$.
By abuse of notation we will denote cylinders simply by $C$ and regard $x,y$ as implicit.
For a point $u\in C$ we set $$ L_C(u) = \{w\in C \;|\;d(w,x)\le d(u,x)\text{ and }d(u,w)\ge 5\theta\}$$ and similarly $$R_C(u) = \{w\in C \;|\;d(w,y)\le d(u,y)\text{ and }d(u,w)\ge 5\theta \}.$$
The \emph{difference} $\diff_C(u,v)$ between $u,v\in C$  is defined by
\begin{multline*}
    \diff_{C}(u,v) = |L_C(u)- L_C(v)| - |L_C(v)- L_C(u)|\\ +|R_C(v)-R_C(u)|-|R_C(u)-R_C(v)|.
\end{multline*}

The following lemma summarizes the main properties of the difference function.  \cite{delzant1995image}]
\begin{lemma}[Lemma 3.4 in \cite{rips1995canonical}]\label{lem: properties of difference}
 If $C$ is a $\theta$-cylinder from $x$ to $y$ then for all $u,v,w\in C$:
\begin{enumerate}[label=(\arabic*)]
    \item if $\bar{C}=(C,y,x)$ denotes the reverse $\theta$-cylinder from $y$ to $x$ then $\diff_{\bar{C}} (u,v)= -\diff_C(u,v)$. 
    \item $\diff_C(u,v)+\diff_C(v,w)=\diff_C(u,w)$ and $\diff_C(u,v)=-\diff_C(v,u)$.\qed
\end{enumerate}
\end{lemma}

It follows that the relation $u\sim v\in C$ if $\diff_C(u,v)=0$ is an equivalence relation. The equivalence classes of this relation are called \emph{slices} and are denoted by $[u]_C$. The relation $[x]_C\prec[y]_C$ if $\diff(x,y)<0$ is well-defined and gives a well-ordering of the slices of $C$.

\begin{lemma}[Proposition 3.6 in \cite{rips1995canonical} and Lemme I.1 in \cite{delzant1995image}]\label{lem: properties of slices}
If $C$ is a $\theta$-cylinder from $x$ to $y$ then for all $u,v,w\in C$:
\begin{enumerate}[label=(\arabic*)]
    \item \label{lem: properties of slices - bounded slices} The slices have a bounded diameter, namely 
        \begin{equation*}
            \diam([u]_C) \le 10\theta.
        \end{equation*} 
    \item \label{lem: properties of slices - q.i of slices} for every two slices $[u]_C\prec[v]_C$ we have
        \begin{equation*}
            \frac{1}{10\theta} \#\{ [w]_C \;|\; [u]_C \preceq [w]_C \prec [v]_C \} \le d(u,v).\qed
        \end{equation*}
\end{enumerate}
\end{lemma}

The next lemma shows that when two cylinders coincide on a large ball their slices coincide, and the order $\prec$ is preserved.

\begin{lemma}[Proposition 3.6 in \cite{rips1995canonical} and Lemme I.1 in \cite{delzant1995image}]\label{lem: stability of slices}
 If $C,C'$ are $\theta$-cylinders between $x,y$ and $x',y'$ respectively, for all $R$, if $w\in C\cap C'$ and $C\cap B_w(R+20\theta)=C'\cap B_w(R+20\theta)$ then 
    \begin{enumerate}[label=(\arabic*)]
        \item (stability of the difference) $\diff_C(u,v)=\diff_{C'}(u,v)$ for all $u,v\in C\cap B_w(R)=C'\cap B_w(R)$.
        \item (stability of slices) in particular $[u]_C=[u]_{C'}$ if $u\in C\cap B_w(R)=C'\cap B_w(R)$.
        \item (stability of the ordering) If in addition $x=x'$ then for all $u,v\in C\cap C'$ if $[u]_C\prec [v]_C$ in $C$ and $[u]_C=[u]_{C'}$ and $[v]_C=[v]_{C'}$ then $[u]_{C'}\prec [v]_{C'}$ in $C'$.\qed
    \end{enumerate}
\end{lemma}

\paragraph{Global stability.}
A $\delta$-hyperbolic graph $X$ with a group action $G\actson X$ \emph{admits $G$-globally stable cylinders} if
there is function $C:X\times X\to 2^{X}$ and $\theta,\tau>0$ such that: 
\begin{itemize}
\item (\emph{$\theta$-cylinders}) $C(x,y)$ is a $\theta$-cylinder from $x$ to $y$ for all $x,y\in X$;
\item (\emph{$G$-invariance}) $gC(x,y)=C(gx,gy)$ for all $g\in G$ and $x,y\in  X$;
\item (\emph{inversion invariance}) $\bar{C}(x,y)=C(y,x)$ for all $x,y\in  X$; and
\item (\emph{$\tau$-stability}) for all $x,y,z\in X$ there exists a finite set $F\subseteq \tilde{X}$ of size $|F|\le \tau$ such that $C(x,y)\cap B - F = C(x,z)\cap B - F$ where $B$ is the ball of radius $$(y.z)_x= \frac{1}{2}(d(x,y)+d(x,z)-d(y,z))$$ around $x$. See Figure \ref{fig: stable cylinders}.
\end{itemize}
A hyperbolic group \emph{admits globally stable cylinders} if its Cayley complex with respect to some (hence every) presentation admits gloablly stable cylinders.

\begin{figure}[ht]
    \centering
    \includegraphics[width=0.75\textwidth]{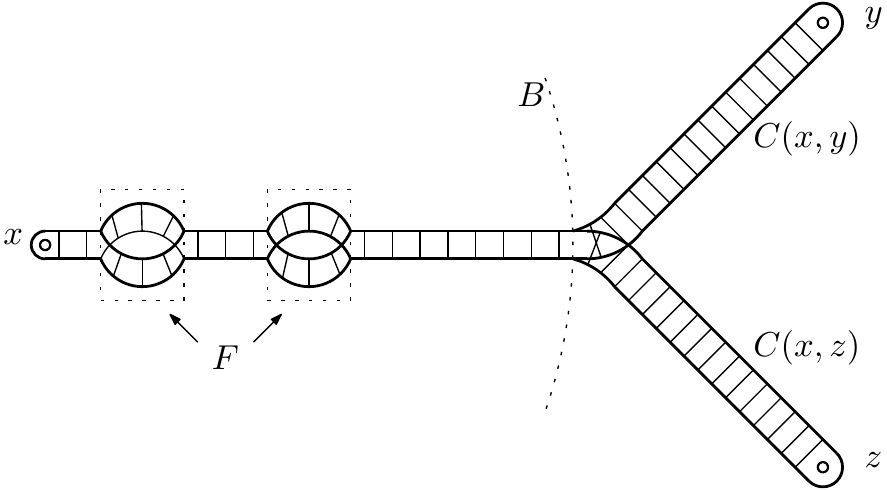}
    \caption{Stable cylinders and their slices.}
    \label{fig: stable cylinders}
\end{figure}

It is shown in the appendix of \cite{rips1995canonical} that $C'(1/8)$ groups admit globally stable cylinders. In \cite{stabilityforCCC} it is shown that hyperbolic cubulated groups admit globally stable cylinders.

\medskip
\begin{lemma}\label{lem: slices in triangles}
If $C: X\times  X\to 2^{ X}$ is globally stable then
\begin{enumerate}[label=(\arabic*)]
    \item\label{lem: slices in triangles - edges} For all $x,y\in X$ the slices of $C(x,y)$ coincide with those of $C(y,x)$, but are reversely ordered.
    \item \label{lem: slices in triangles - triangles} there is a constant $\epsilon = \epsilon(\delta,\theta,\tau)$ with the following property.
    For every $x,y,z\in X$, if we denote by $n$ the number of slices in $C(x,y)$ then there exists $1\le k\le n$ such that with the exception of at most $\epsilon$ slices the first $k$ slices of $C(x,y)$ are equal to the first $k$ slices of $C(x,z)$ and are ordered in the same way, and the last $n-k$ slices of $C(x,y)$ are equal to the last $n-k$ slices of $C(z,y)$ and are ordered in the same way.
\end{enumerate}
\end{lemma}
\begin{proof}
    It follows from the definition of global stability and Lemma \ref{lem: stability of slices}.
\end{proof}


This motivates the definition of the following singular foliation for groups acting on $X$.

\section{Delzant's singular foliation and accessibility}\label{sec: the singular foliation}
Let $G$ be a one-ended group with a finite presentation $G=\gen{x_1,\ldots,x_s | u_1,\ldots,u_t}$ in which all the relations are of length $\le 3$. 
Let $\phi:G\to \Aut(X)$ be a free action of $G$ on $X$.
Assume that $X$ admits $G$-globally stable cylinders $C:X\times X\to 2^X$.
Delzant \cite{delzant1995image} constructs a (discrete) singular foliation on the presentation complex $P_G$ of $G$ which we slightly adapt to our setting in the following way.\footnote{In \cite{delzant1995image} $\phi$ is not assumed to be an embedding, and $G$ is not assumed to be one-ended. However in our application this is the case. On the other hand, Delzant assumes a different stability property which is not the same as the global stability discussed above.}

We begin by describing a singular foliation on the universal cover $\tild{P}_G$.
The vertices $\tild{P}_G ^{(0)}$ of $\tild{P}_G$ are in one-to-one correspondence with elements of $G$. 
So we write $\tild{P}_G ^{(0)}=G$. 
Fix a vertex $x\in X^{(0)}$, and define the map $\Phi_x:\tild{P}_G ^{(0)} =G\to X^{(0)}$ by $\Phi_x(g) = \phi(g)x$.
For each (oriented) edge $e\subset \tild{P}_G^{(1)}$ with endpoints $i(e),t(e)$, map $\Phi_x(e)$ equivariantly to the geodesic in $X$ connecting $\Phi_x(i(e)), \Phi_x(t(e))$. Let $C_{x}(e) = C(\Phi_x(i(e)),\Phi_x(t(e)))$ be the cylinder in $X$. 
We mark equally spaced points on the edge $e$ one for each slice of $C_x(e)$ ordered according to the order in which they appear in $C_x(e)$ (as described above). 
Note that by Lemma \ref{lem: slices in triangles}.\ref{lem: slices in triangles - edges} the assignment of a marked point on the edge to the slice in $ X$ does not depend on the orientation of the edge $e$.
In each triangle (or bigon) $\Delta\subset \tild{P}_G^{(2)}$, we consider the following disjoint collection of closed arcs (see Figure \ref{fig: singular foliation}):
\begin{itemize}
    \item (\emph{regular arcs}) We connect with an arc the marked points on the edges if their corresponding slices in $X$ coincide. By Lemma \ref{lem: slices in triangles}.\ref{lem: slices in triangles - triangles}, this leaves at most $3\epsilon$ not connected.
    \item (\emph{singular arcs}) To each of the remaining points we attach an arc which ends in the interior of the triangle $\Delta$. 
\end{itemize}
Let $\tild{\calF}\subset \tild{P}_G$ be the union of all of these arcs and points.
By the $G$-invariance of the cylinders, we can choose $\tild{\calF}$ to be $G$-invariant.

\begin{figure}[ht]
    \centering
    \includegraphics[width=0.25\textwidth]{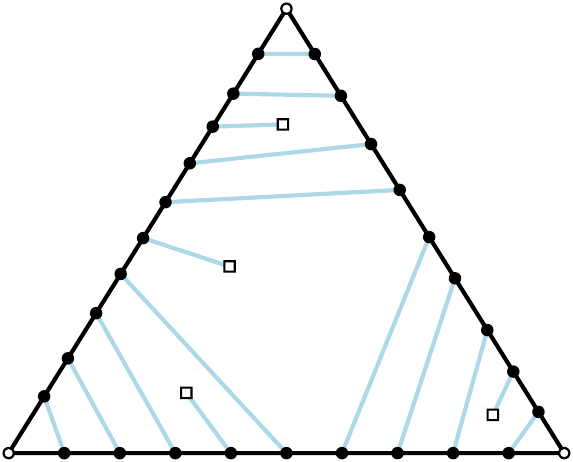}
    \caption{The singular foliation on a triangle $\Delta\subseteq\tild{P}_G$.}
    \label{fig: singular foliation}
\end{figure}

Let $\calF=\tild{\calF}/G$ be the singular foliation on $P_G$.
We call the connected components of $\tild{\calF}$ (resp. $\calF$) the \emph{leaves} of this foliation.
Each leaf of $\tild{\calF}$ (resp. $\calF$) is associated with a slice in a cylinder in $X$ (resp. a $\phi(G)$-orbit of slices in $X$). 
We note that the map that maps a leaf of $\calF$ (resp. $\tild{\calF}$) to its corresponding slice in $X$ (resp. $\phi(H)$-orbit of slices in $X$) is not necessarily injective.

Our goal is to show that up to conjugating $\phi$ we can bound the number of leaves in the foliation $\calF$ by a linear function of $s$ and $t$. More explicitly,


\begin{lemma} \label{lem: accessibility upper bound}
    Let $G$ be a hyperbolic group with globally stable cylinders. There exists $\omega = \omega(\delta,\theta,\tau)\ge  0$ such that if $G=\gen{x_1,\ldots,x_s | u_1,\ldots,u_t}$ is one-ended, and all relations are of length $\le 3$, and if $\phi:G \to \Aut(X)$ is a free action. Then there exists $x\in X$ such that the number of leaves in the singular foliation $\calF$ induced by $\Phi_x$ on the presentation complex $P_G$ of $G$ is bounded by $ \omega \cdot t$.
\end{lemma} 

\begin{remark}
\begin{enumerate}
    \item Note that Lemma \ref{lem: accessibility upper bound} does not claim that the length of $\phi(x_i)$ is bounded. Each edge of $P_G$ can potentially traverse the same leaf of $\calF$ many times.
    \item Changing the base point $x$ and applying a left translation by an element of $\Aut(X)$ amounts to conjugating $\phi$ by an element of $\Aut(X)$. Compare Lemma \ref{lem: accessibility upper bound} with the main theorem of \cite{delzant1995image}.
\end{enumerate}
\end{remark}

\begin{proof}
We will do so by dividing the task into the different types of leaves of $\calF$:
\begin{description}
\item[Type I] The leaf does not contain a singular arc and it is two sided.
\item[Type II] The leaf does not contain a singular arc and it is one-sided.
\item[Type III] The leaf contains a singular arc.
\end{description}


\paragraph{Type I.} 

To bound the number of Type I leaves, let us consider the following graph of groups $\calG$ which is a variant of the one constructed in \cite{delzant1995image}.
Let $\tild{\calF}_{\I},\calF_{\I}$ be the set of type I leaves of $\tild{\calF},\calF$ respectively.
We will assume that $\calF_{\I}\ne \emptyset$ as otherwise there is nothing to bound.
Since all the leaves of $\tild{\calF}_{\I}$ are two sided, each of them separates $\tild{P}_G$ into 2 components. 
Let $T$ be the tree whose vertices are the components of $\tild{P}_G - \tild{\calF}$ and whose edges are the components of $\tild{\calF}$. 
See Figure \ref{fig: foliation tree}.
Clearly, $G\actson T$.
Note that by construction, the leaves of $\tild{\calF}$ correspond to slices in $X$. Therefore, the edge-stabilizers of $G\actson T$ are finite since slices are bounded (Lemma \ref{lem: properties of slices}.\ref{lem: properties of slices - bounded slices}).
The associated graph of groups $\Gamma$ has the connected components of $P_G - \calF_{\I}$ as vertices, the leaves of $\calF_{\I}$ as edges, and the image of their fundamental group in $\pi_1(P_G)=G$ as vertex/edge groups. We denote by $\pi:T\to\Gamma$ the quotient map.

\begin{figure}[ht]
    \centering
    \includegraphics[width=\textwidth]{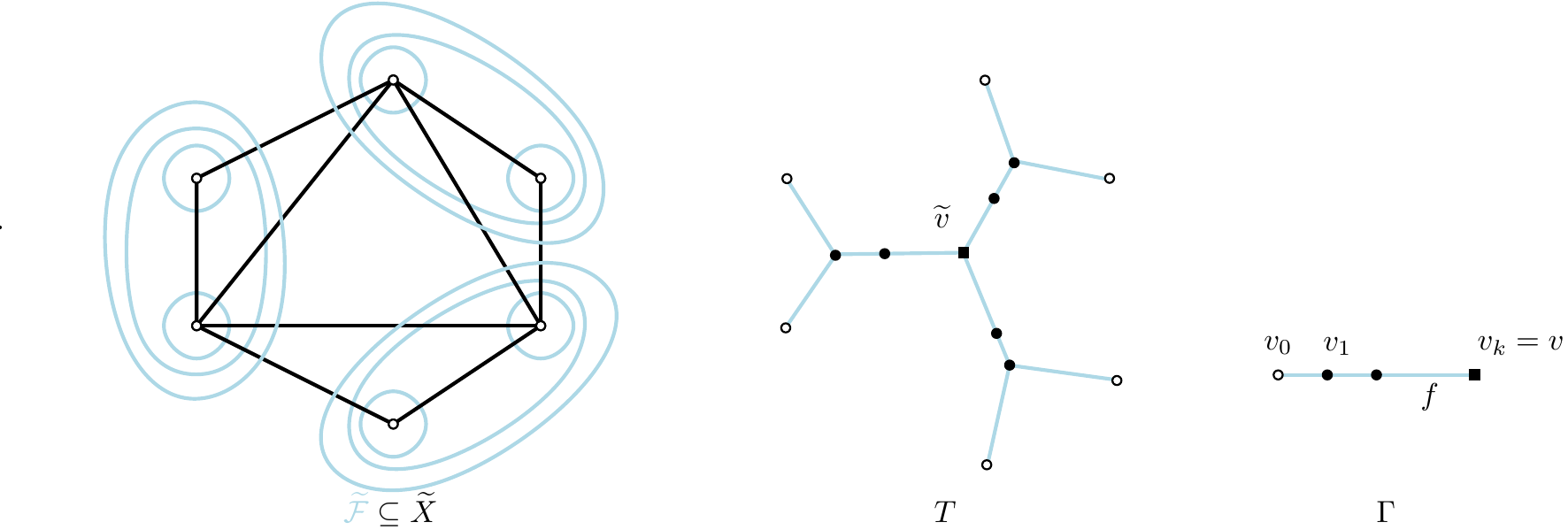}
    \caption{The foliation $\tild\calF_{\I}$ on $P_G$, the dual tree $T$, and the graph of groups $\Gamma$}
    \label{fig: foliation tree}
\end{figure}

Since $\pi_1(\Gamma) = G$ and $G$ is one-ended, the graph of groups $\Gamma$ describes a trivial splitting. I.e, $G\actson T$ has a global fixed point $\tild{v}$.
Equivalently, the underlining graph of $\Gamma$ is a tree, there is  a vertex $v=\pi(\tild{v})\in \Gamma$ whose vertex group is $\Gamma_v=G$, and for any edge $a\in \Gamma$ if $w$ is the endpoint of $a$ that is further away from $v$, then $\Gamma_w = \Gamma_a$.

There is a natural map $\tild\psi : \tild{P}_G \to T$, and similarly $\psi:P_G \to \Gamma$.
Clearly, both maps are onto. For every $e\in \tild{P}_G$, the marked points on $e$ correspond to distinct slices, therefore $\tild\psi(e)$ is a geodesic in $T$.
It follows that the underlining graph of $\Gamma$ is an interval with one endpoint $v$ and the other is $v_0=\psi(\ast)$ where $\ast\in P_G$ is the unique vertex of the presentation complex $P_G$. Moreover, if we denote by $v_0,\ldots,v_k=v$ the vertices of $\Gamma$ in the order they appear along the path from $v_0$ to $v$ then for every edge $e\subset P_G$, $\psi(e)$ is a path in $\Gamma$ of the form $v_0,v_1\ldots v_{m-1},v_m,v_{m-1},\ldots,v_1,v_0$.
The vertex $v$ has only one lift $\tild{v}$ to the tree $T$, i.e $\{\tild{v}\}=\pi\ii (v)$. Let $f$ be the unique edge incident to $v$ in $\Gamma$.
For every $g\in G = \tild{P}_G ^{(0)}$, let $\tild{f}_g$ be the unique edge in $\pi \ii (f)$ separating $\tild{\psi}(g)$ and $\tild{v}$, and let $s_g$ be the associated slice in $X$.

If we change the base point $x$ (in the definition of $\Phi_x$) to any point $x'$ in $s_1$. By $\phi$-equivariance it follows that $\Phi_{x'}(g) \in g s_1 = s_g$. For every $\tild{e}\in \tild{P}_G^{(1)}$ one of two things has happened:

Case 1. The path $\tild{\psi}(\tild{e})$ did not pass through $\pi\ii(f)$. Then both endpoints of $\tild{e}$ are sent by $\Phi_{x'}$ to the same slice, and since $\Phi_{x'}(\tild{e})$ is sent to the geodesic between them, its length is bounded by the diameter of a slice, i.e $\length(\Phi_{x'}(\tild{e}))\le 20\theta$. 

Case 2. The path $\tild{\psi}(\tild{e})$ passed through $\pi\ii(f)$. The endpoints of $\Phi_{x'}(\tild{e})$ are in slices of the cylinder $C_{x,e}$, and thus are at distance $\theta$ from the geodesic $\Phi_{x}(\tild{e})$. By Lemma \ref{lem: properties of slices}.\ref{lem: properties of slices - q.i of slices} the endpoints of $\Phi_{x'}(\tild{e})$ are at distance at least $\frac{k}{10\theta}$ from the endpoints of $\Phi_{x}(\tild{e})$. Thus, \[\length(\Phi_{x'}(\tild{e})) \le \length(\Phi_x(\tild{e})) - 2\frac{k}{10\theta} + 2 \theta.\]

If we assume that $x$ is chosen so that \[\max\{\length(\Phi_x(\tild{e}))| \tild{e}\subseteq \tild{P}_G^{(1)}\}\] is minimal, then each edge has at most $\epsilon'' = 10\theta^2$ regular leaves of the foliation $\calF$.
Therefore, we may assume that $\calF$ has at most $s\epsilon''$ Type I leaves.
Since $G$ is one-ended every generator belong to an relation, and so $s\le t/3$. Thus $\calF$ has at most  $\frac{\epsilon''}{3}t$ Type I leaves.


\paragraph{Type II.}
Dunwoody \cite{dunwoody1985accessibility} shows that the number of type II leaves is bounded by $b_1(G)$ where $b_1$ is the first betti number of $G$ (and therefore of $P_G$), which in turn is bounded by $s$ -- the number of generators of $G$ (and therefore the number of 1-cells of $P_G$). 

In fact, Type II cannot occur in our setting. A Type II leaf $\ell$ corresponds to a non-trivial free product with amalgamation $G=A*_C B$ with $A=i_*\pi_1(\ell), B=i_*\pi_1(P_G - \ell), C=i_*\pi_1(N(\ell)-\ell)$ where $i$ is the inclusion, and $N(\ell)$ is some small regular neighborhood of $\ell$.
Clearly, $C$ is finite since $\phi(C)$ stabilizes a slice, and all slices have a bounded diameter. It is easy to see that $[A:C]=2$, and hence the splitting is non-trivial.

\paragraph{Type III.}

Each slice of Type III must contain a singular arc. There are at most $3\epsilon$ singular arcs in a triangle $\Delta$. 
There are $t$ 2-cells in $P_G$, and therefore the number of Type III leaves of $\calF$ is bounded by $ 3\epsilon \cdot t$.

\bigskip

\noindent Combining the above bounds we get that $\omega = 3\epsilon + \frac{\epsilon''}{3}$ is the desired constant.
\end{proof}

\section{Boundaries and Cohomology}

\paragraph{Finiteness properties and complexity.}
Let $m\in \bbN$. A group $G$ has type $F_m$ if $G$ acts freely and cocompactly on an $(m-1)$-connected $m$-dimensional $\Delta$-complex $R=R(m,G)$.
For $i\le m$ we define $\Cx_{i,m}(G)$ to be the minimal number of $G$-orbits of cells of dimension $\le i$ in such an action\footnote{Note that we assume that $R$ is a $\Delta$-complex and not just a CW-complex.}.
In particular, $\Cx_{1,1}(G) = \rank(G)$ and (if $G$ is torsion free) $\Cx_{2,2}(G)=T(G)$ defined in the introduction.

Let $X$ be a graph. For $d\in\bbN$, the Rips complex $R_d(X)$ is the simplicial complex whose simplices are subsets of the vertex set of $X$ of diameter $\le d$. 
If $X$ is $\delta$-hyperbolic then $R_d(X)$ is contractible for $d\ge 4\delta+2$.
Thus, torsion-free hyperbolic groups are of type $F$. 
Modifying slightly the construction of Rips complexes, see \cite{bestvina1991boundary}, we see that hyperbolic groups are of type $F_\infty$, i.e of type $F_m$ for all $m$. 
Hence, if $G$ is hyperbolic then $\Cx_{i,m}(G)$ is well-defined and finite for all $i\le m$.

We will need the following fact $H^i(G;\bbZ G) = H^i_c(R(m,G);\bbZ)$ for all $i< m-1$. In fact, one can replace $R(m,G)$ with any $(m-1)$-connected complex on which $G$ acts properly and cocompactly.



\paragraph{Uniform quasi-surjectivity of quasi-isometries.}
The next proposition shows that continuous quasi-isometries between hyperbolic groups are, in a sense, uniformly quasi-surjective.

Let $G$ be a hyperbolic group. The boundary $\partial G$ is a compact metrizable space which has a finite topological (covering) dimension, $\dim(\partial G)<\infty$ (see \cite{ghys1990groupes}).
The boundary is a quasi-isometry invariant.

\begin{proposition}\label{prop: uniform quasi-onto}
Let $R$ be a hyperbolic $\Delta$-complex with a proper cocompact group action, and let $m=\dim(\partial R)+1$. 
Assume that $R$ is $m$-dimensional and $(m-1)$-connected.
Then, there exists $r_0$ such that if a group $G$ acts on $R$ properly and cocompactly, and $\Phi:R(m,G) \to R$ is any continuous $G$-equivariant quasi-isometry then $N_{r_0}(\Phi(R(m,G)))=R$.
\end{proposition}

\begin{proof}
Since $R$ admits some cocompact action, it suffices to show that the intersection $\Phi(R(m,G)) \cap B_{r_0}\ne \emptyset$ where $B_{r_0}$ is the ball of radius $r_0$ around some fixed point $x_0$ in $R$.

The orbit map of the action $G\actson R$ is a quasi-isometry which induces a boundary homeomorphism $\partial\Phi : \partial G \to \partial R$.
By Bestvina-Mess \cite{bestvina1991boundary} we have \[m-1=\dim\partial R = \max\{k \;|\; \cechH ^k(\partial R;\bbZ) \ne 0\}\] where $\cechH ^k$ is the reduced \v{C}ech cohomology (see \cite{walsh1981dimension}),
and  $$H^m(G;\bbZ G) = \cechH ^{m-1}(\partial R;\bbZ).$$
By \cite{geoghegan1986note}, one can compute $H^m(G;\bbZ G)$ as the limit over all compact subsets $K\subseteq R$ of the cohomology groups
$H^{m-1}(R-K;\bbZ)$.
Thus,
\begin{equation*}\label{eq: bestvina-mess cohomology}
    \varinjlim_K H^{m-1}(R-K;\bbZ)=\cechH ^{m-1}(\partial R;\bbZ).
\end{equation*}

Let $0\ne \alpha\in \cechH ^{m-1}(\partial R;\bbZ)$, then there exists $r_0$ such that $\alpha \in H^{m-1}(R-B_{r_0};\bbZ)$.
We have the commuting diagram
\[\begin{tikzcd}
&\alpha\in H^{m-1}(R-B_{r_0}) \arrow[d]\arrow[r,"\Phi^*"'] &H^{m-1}(R(m,G)-\Phi\ii(B_{r_0}))\arrow[d]\\
&\alpha\in \cechH ^{m-1}(\partial R;\bbZ)\arrow[r,"\simeq","(\partial \Phi)^*"'] &\cechH ^{m-1}(\partial R(m,G);\bbZ)
\end{tikzcd}
\]
If $\im\Phi\cap B_{r_0}= \emptyset$ then $\Phi\ii(B_{r_0})=\emptyset$, and it follows that \[H^{m-1}(R(m,G)-\Phi\ii(B_{r_0})) = H^{m-1}(R(m,G))=0\] since $R(m,G)$ is $(m-1)$-connected. However, this contradicts that $(\partial \Phi)^*(\alpha) \ne 0$ as $\partial \Phi$ is a homeomorphism.
\end{proof}

\section{Lower bounds on complexity}
\label{sec: proof of thm A}

Theorem \ref{thm: main volume complexity} follows from the following theorem. Note that unlike Theorem \ref{thm: main volume complexity} the following is non-vacuous for hyperbolic groups with torsion.

\begin{theorem}\label{thm: main volume complexity upgraded}
    Let $X$ be a one-ended hyperbolic graph which admits $\Aut(X)$-globally stable cylinders. Then there exists $\alpha=\alpha(X)>0$ such that if a group $G$ acts freely and cocompactly on $X$ then $$\alpha\cdot \Vol_X(G) \le \Cx_{2,m}(G)$$ where $m=\dim(\partial G)+1$.
\end{theorem}

\begin{proof}
Let $\bfG = \Aut(X)$.
By replacing $X$ by the 1-skeleton of its Rips complex, we may assume that there is a contractible $\bfG$-complex $R_X$ with $\bfG$-globally stable cylinders such that $X=R^{(1)}$.

If $G\actson X$ freely and cocompactly then $G\actson R_X$ properly cocompactly, and acts freely on its vertices.
Let $R_G=R(m,G)$ be the complex realizing $\Cx_{2,m}$, that is, an $m$-dimensional $(m-1)$-connected $\Delta$-complex with a free cocompact $G$-action and $\Cx_{2,m}(G)$ many $G$-orbits of 2-cells.
By contracting a spanning tree in $R_G/G$ we may assume that it has a single orbit of vertices.
Let $P_G = R_G^{(2)}/G$ be its 2-skeleton. Since $R_G^{(2)}$ is simply connected and $G\actson R_G^{(2)}$ freely we know that $\pi_1(P_G)=G$. Since we also assume $P_G$ has only one vertex, $P_G$ is the presentation complex of some presentation $$G=\gen{x_1,\ldots,x_s | u_1,\ldots,u_t}$$ in which all relations have length $\le 3$ and $t=\Cx_{2,m}(G)$.




The embedding $\phi:G\into \Aut(X)$ endows its presentation complex $P_G$ with the singular foliation described in \S\ref{sec: the singular foliation}. 
By Lemma \ref{lem: accessibility upper bound}, there exists $x\in X$, such that the singular foliation $\calF$ on $P_G$ induced by $\Phi_x$ satisfies
\begin{equation}\tag{I}\label{upper bound}
    \# \{\text{leaves in }\calF\} \le \omega\cdot t
\end{equation}

Recall that $\Phi_x:R_G^{(1)}=\tild{P_G}^{(1)} \to X=R^{(1)}$ was defined by sending the vertices $R_G^{(0)}$ equivariantly to $R$, and sending edges equivariantly to the geodesics connecting them.
One can extend $\Phi_x:R_G^{(1)} \to R^{(1)}$ to a continuous $\phi$-equivariant map $\Phi:R_G \to R$ inductively in the following way. 
For all $n\le m+2$, using the $n$-connectivity of $R$, map an $n$-cell $\sigma$ of $R_G$ ($\phi$-equivariantly) to the filling $n$-disk of $\Phi|_{R_G^{(n-1)}}(\partial\sigma)$ in $R$. 
By \cite[Lemma 1.7.A]{gromov1987hyperbolic}\cite[\S4.2 Proposition 9]{ghys1990groupes}, this can be done in the $r_1$-neighborhood of $\Phi(R_G^{(1)})$, where $r_1$ depends only on $R$. 



By Proposition \ref{prop: uniform quasi-onto} there exists a number $r_0$ such that $N_{r_0}(\Phi(R_G)) = R$. 
By the definition of $r_0,r_1$ and the foliation $\tild{\calF}$, we have the following.
\begin{itemize}
\item Every $r_0$-ball in $R$ contains a point of $\Phi(R_G)$.
\item Every $r_1$-ball around a point of $\Phi(R_G)$ contains a point of $\Phi(R_G^{(1)})$.
\item Every point in $\Phi(R_G^{(1)})$ is contained in a slice corresponding to a leaf of the foliation $\tild{\calF}$. Since, by Lemma \ref{lem: properties of slices}.\ref{lem: properties of slices - bounded slices}, slices have diameter $\le 10\theta$, every $10\theta$ ball around a point in $\Phi(R_G^{(1)})$ contains a slice.
\end{itemize}
Setting $\rho=r_0+r_1+10\theta$ we get that every $\rho$-ball in $X$ contains a slice corresponding to a leaf of the singular foliation $\tild{\calF}$.

There are $\Vol_X(G)$ many $G$-orbits of points in $X$.
If $\nu$ is the number of points in a ball of radius $2\rho$ in $X$, then there are at least $\Vol_X(G) / \nu$ disjoint $\rho$-balls in $X/G$ and therefore 
\begin{equation}\label{lower bound} \tag{II}
      \frac{1} {\nu} \Vol_X(G) \le \#\{\text{ leaves in }\calF\}
\end{equation}

Combining \eqref{upper bound} and \eqref{lower bound} we get the desired inequality
\[\alpha \Vol_X(G) \le t=\Cx_{2,m}(G),\]
where $\alpha := \frac{1}{\nu\omega}>0$ depends only on $X$.
\end{proof}

We now turn to the proof of Theorem \ref{thm: main sqrt bound}. Again, one can phrase a version of Theorem \ref{thm: main sqrt bound} which uses $\Cx_{2,m}(G)$ instead of $\Cx(G)$. Since the proof follows similar lines as the proof above. Instead of repeating it, we highlight the main differences. 

\begin{proof}[Proof of Theorem \ref{thm: main sqrt bound}]
\begin{itemize}
    \item Replace the globally stable cylinders in the proof above by the cylinders constructed by Rips-Sela \cite{rips1995canonical} for proper cocompact actions of a finitely presented group $G=\gen{x_1,\ldots,x_s|u_1,\ldots,u_t}$ on a hyperbolic graph $X$. To be precise, Rips-Sela construct such stable cylinders for maps between a finitely presented group and a hyperbolic group, but the construction carries through to our setting. It has the same properties as globally stable cylinders, except that $C(\cdot,\cdot)$ is only defined for pairs of points in $X$ that differ by the action of one of the generators of $G$, the second crucial difference is that in Lemma \ref{lem: slices in triangles}.\ref{lem: slices in triangles - triangles} instead of a constant $\epsilon$ we have $\epsilon = O(t)$.
    \item Construct the singular foliation as in Delzant \cite{delzant1995image}. Note that for each 2-cell $\Delta$, the number of singular arcs in $\Delta$ is now $O(t)$.
    \item It immediately follows that the bound on Type III leaves is $O(t^2)$. Leaves of Type I and II are still bounded as in the proof of Lemma \ref{lem: accessibility upper bound}. Thus, the number of leaves in the singular foliation is now bounded by $O(t^2)$.
\end{itemize} 
The rest of the proof is the same, and yields $\Vol_X(G) = O(t^2)$.
\end{proof}

\begin{theorem}\label{thm: main upgraded}
Let $G$ be a one-ended hyperbolic group, let $m=\dim(\partial G)+1$. Then, there exist $\alpha,\beta>0$ such that if $H\le G$ is a finite index subgroup then $$\alpha\cdot \sqrt{[G:H]} \le \Cx_{2,m}(H) \le \beta \cdot [G:H].$$

If moreover $G$ admits globally stable cylinders then $$\alpha\cdot [G:H] \le \Cx_{2,m}(H) \le \beta \cdot [G:H].$$
\end{theorem}
\begin{proof}
Assume $G$ is a one-ended hyperbolic groups. Let $R_G=R(m,G)$ be a $\Delta$-complex with $\Cx_{2,m}(G)$ $G$-orbits of 2-cells.
We may assume that $R_G$ has one $G$-orbit of vertices.
$G$ acts on $R$ freely, and so does every finite index subgroup $H\le G$.
It is immediate that $R$ is an $R(m,H)$ with $[G:H]\cdot \Cx_{2,m}(G)$ 2-cells, and thus $$\Cx_{2,m}(H) \le [G:H]\cdot \Cx_{2,m}(G)$$ giving us the desired upper bound.
 
 Let $X=R_G^{(1)}$. The lower bound follows from Theorems \ref{thm: main sqrt bound} and \ref{thm: main volume complexity} after observing that $\Vol_X(H) \asymp |R_G^{(0)}/H|= [G:H]$. 
\end{proof}

The lower bounds in Theorem \ref{thm: main} follow from the theorem above, while the upper bounds are proved similarly replacing $R(m,G)$ by the $K(G,1)$ with minimal number of cells.

\section{Multi-ended groups and Theorem \ref{thm: free splitting}}
\label{sec: multiended}

Let the group $G$ act on the set $X$ with finitely many orbits $X=Gx_1 \coprod \ldots \coprod Gx_n$, define \[V_X(G)=\sum_{i=1}^n \frac{1}{|G_{x_i}|},\] where if $G_{x_i}$ is infinite then $\frac{1}{|G_{x_i}|} = 0$.

\begin{lemma}
    If $G\actson X$, and $H\le G$ is a subgroup of finite index then $V_X(H) = [G:H]\cdot V_X(G)$.
\end{lemma}

\begin{proof}
    Without loss of generality, we may assume that $G\actson X$ is transitive and with finite stabilizers. 
    Let $x\in X$, and let $G_x = K$ be its finite stabilizer. Identify $X=G/K$.
    Thus, the equality we wish to show becomes
    \begin{equation} \label{eq: double cosets lemma}
    [G:H] = \sum_{i=1}^m \frac{|K|}{|\Stab_H(g_iK)|},
    \end{equation}
    where $g_1K,\ldots,g_mK$ are representatives for the $H$-orbits on $G/K$. In other words, they are representatives for the double cosets $G = H g_1 K \coprod \ldots \coprod H g_m K$.
    
    The summands on the right hand side of \eqref{eq: double cosets lemma} are
    \begin{align*}
        \frac{|K|}{|\Stab_H(g_iK)|} &=  \frac{|K|}{|H\cap g_i K g_i\ii|} \\
        &=\frac{|K|}{|g_i\ii H g_i\cap K |}\\
         &=  [K:g_i\ii H g_i\cap K ]\\
        &= [K:\Stab_K(Hg_i)]
    \end{align*} 
    where the  the stabilizer in the last expression is with respect to the action $H\backslash G \actsno K$.
    The sum $ \sum_{i=1}^m \frac{|K|}{|\Stab_H(g_iK)|} $
    is thus the sum of the size of all orbits of $H\backslash G \actsno K$ which is therefore equal to $|H\backslash G| = [G:H]$ as desired.
\end{proof}

\begin{corollary}\label{lem: euler char of splittings}
    Let $\calG$ be a finite graph of groups with finite edge stabilizers, define $\chi(\calG)= \sum_{v} \frac{1}{|G_v|} - \sum_{e} \frac{1}{|G_e|}$. If $H\le G$ has finite index, and if $\calH$ is the graph of groups corresponding to $H$, then, $\chi(\calH) = [G:H]\cdot \chi(\calG)$.
\end{corollary}
\begin{proof}
Apply the previous lemma to the vertices and edges of the Bass-Serre tree of $\calG$.
\end{proof}

By Dunwoody \cite{dunwoody1985accessibility}, every finitely presented group $G$ admits a \emph{maximal} splitting over finite groups, i.e, $G$ is the fundamental group of a graph of groups whose edge groups are finite and whose vertex groups have at most one end (i.e, finite or one-ended).

\begin{lemma}\label{lem: euler char of splitable}
Let $G$ be a finitely presented group, and $G=\pi_1(\calG)=\pi_1(\calG')$ where $\calG$ and $\calG'$ are maximal splittings for $G$. Then, $\chi(\calG)= \chi(\calG')$.
\end{lemma}

\begin{proof}
   Let $\calG$ and $\calG'$ be two maximal splittings. Every vertex group of $\calG$ has a fixed point in $\calG'$, this allows us to construct a map between the graph of groups and shows that one can get from $\calG$ to $\calG'$ by subdividing the edges of $\calG$ and then performing Stallings' folds. 
    
    Subdividing edges clearly does not change the $\chi$, as it creates both a vertex and an edge with the same group.
    
    It suffices to show that $\chi$ is preserved if $\calG'$ is obtained from $\calG$ by an elementary Stallings' fold. There are several types of folds that can happen. Let $T$ be the Bass-Serre tree of $\calG$, and let $\pi:T\to\calG$ be the quotient map. Let $e_1,e_2$ be the two folded edges in $T$, which share a vertex $v$, and whose other endpoints are $w_1,w_2$ respectively. 
    Using the conventions of \cite{bestvina1991bounding} we divide into cases.
    
    \paragraph{Type A. $\pi(v)\ne \pi(w_i)$ for $i=1,2$.} Here there are 3 cases:

    Type IA. $\pi(w_1)\ne\pi(w_2)$. Then necessarily $G_{w_1} = G_{e_1}$ or 
    $G_{w_2}=G_{e_2}$ as otherwise the fold would produce a vertex group with more than one end. It follows that $\chi(\calG) = \chi(\calG')$.
    
    Type IIA. There is an element $t\in G_v$ such that $te_1 = e_2$. Then again necessarily $G_{e_1}=G_{w_1}$ and it follows that $\chi(\calG) = \chi(\calG')$.
    
    Type IIIA. There is an element $g\in G - G_v$ such that $gw_1=w_2$. This cannot occur, as it will produce a vertex group with more than one end. 
    
    \paragraph{Type B/C. $\pi(v)=\pi(w_1)$.} There are 4 cases to consider:

    Type IB. $\pi(v)\ne \pi(w_2)$. This case is similar to Type IA. 
    
    Type IIB. There is an element $t\in G_v$ such that $te_1 = e_2$. This case is similar to Type IIA. 
    
    Type IIIB. There is an element $g\in G$ such that $gv = w_2$. This is again impossible.
    
    Type IIIC. There is $g\in G- G_v$ such that $ge_1=e_2$. This fold can be realized as a subdivision of the edge and two folds of the types we have already discussed.
\end{proof}

Therefore, we can define $\chi(G) := \chi(\calG)$ for some (hence any) maximal graph of group $\calG$ for $G$.

\begin{proof}[Proof of Theorem \ref{thm: free splitting}]
    First we observe that if $\chi(G)\ne 0$ then $G$ cannot contain two isomorphic subgroups of different indices. Assume that $A,B\le G$ are isomorphic subgroups of indices $a,b$ respectively. Then, we have
    \[ a \cdot \chi(G)=\chi(A)=\chi(B)=b\cdot \chi(G)\]
    where the first and last inequalities follow from Lemma \ref{lem: euler char of splittings} and the middle equality follows from Lemma \ref{lem: euler char of splitable}. Since $\chi(G)\ne 0$ we get that $a=b$.
    
    Therefore, it suffices to show that if $G$ has more than two ends then $\chi(G)<0 $.
    Let $\calG$ be a graph of groups with $\pi_1(\calG)= G$ and every vertex group has at most one end.
    We may assume without loss of generality that the graph of groups is reduced, that is, if the vertex $v$ is incident to the edge $e$ and $G_e=G_v$ then $e$ is a loop at $v$.
    For a vertex $v\in \calG$, we define \[\chi_+(v) = \frac{1}{|G_v|} - \frac{1}{2}\sum_{v\in e} \frac{1}{|G_e|}.\]
    We clearly have $\chi(\calG) = \sum_v \chi_+(v).$
    Note that $|G_v| \ge |G_e|$ and therefore $\chi_+(v)>0$ if and only if $d(v)\le 1$ and if $d(v)=1$ then $G_v=G_e$. If $d(v)=0$ then $G=G_v$ and the group has at most one end. If $d(v)=1$ and $G_v=G_e$ then the graph of groups is not reduced.
    Thus, $\chi_+(v)\le 0$. 
    If $\chi(\calG)=0$ then $\chi_+(v)=0$ for all $v\in \calG$.
    $\chi(v)=0$ happens in the following three cases.
    
    Case 1. $d(v)=2$, there are two distinct edges $e_1,e_2$ incident to $v$, and one has $|G_v| = |G_{e_1}|=|G_{e_2}|$. In this case, the graph of groups is not reduced.
    
    Case 2. $d(v)=2$, and there is only one edge $e$ incident to $v$ forming a loop, and $|G_v| = |G_{e_1}|=|G_{e_2}|$. In this case, the fundamental group of the graph of group is virtually cyclic.
    
    Case 3. $d(v)=1$ and for the edge $v\in e$, one has $|G_v|=2\cdot|G_e|$. The other end $w$ of the edge $e$ must be of the same form, and again, the graph of groups describes a virtually cyclic group.
\end{proof}

\section{Abstract commensurators}\label{sec: abstract comm}

Two isomorphisms $\phi_i:A_i\to B_i$, $i=1,2$, of finite-index subgroups $A_i,B_i$ of $G$ are \emph{equivalent}, $\phi_1\sim \phi_2$, if there exists $A\le A_1\cap A_2$ of finite index such that $\phi_1|_A = \phi_2|_A$.
The \emph{abstract commensurator} of $G$ is the group $\Comm(G)$ whose elements are equivalence classes of isomorphisms between finite-index subgroups of $G$, and group multiplication of $\phi:A_\phi\to B_\phi,\psi:A_\psi\to B_\psi$ is the restriction of the composition $\psi\circ \phi$ to its natural domain $\phi\ii(B_\phi \cap A_\psi)$.

Given an abstract commensurator $\phi:A\to B$ it is natural to consider powers $\phi^n$ of $\phi$ as abstract commensurators.
The power $\phi^n$ is naturally a map from $A_n\to B_n$ where $A_n,B_n$ are defined recursively by
\begin{align*}
    &A_1=H, & &B_1=B, \\
    &A_i=\phi\ii(B\cap A_{i-1}), & &B_i=\phi(A\cap B_{i-1})
\end{align*}
for all $i>1$.

We get the following lattice of subgroups.
\[\begin{tikzcd}[column sep = tiny]
& &      &       &G \ar[dl,dash, "a_1" description]\ar[dr,dash," b_1" description]   &       & 
\\
& &      &A_1\ar[rr,blue,"\phi"]\ar[dl,dash,"a_2" description]\ar[dr,dash," b_2" description]  &       &B_1\ar[dl,dash,"a_2"description]\ar[dr,dash," b_2"description]   & 
\\
& &A_2 \ar[rr,blue,"\phi"]\ar[dl,dash, "a_3"description]\ar[dr,dash," b_3"description] & &A_1\cap B_1 \ar[rr,blue,"\phi"]\ar[dl,dash, "a_3"description]\ar[dr,dash," b_3"description]& &B_2 \ar[dl,dash, "a_3"description]\ar[dr,dash," b_3" description]
\\
&A_3 \ar[rr,blue,"\phi"] & & A_2\cap B_1 \ar[rr,blue,"\phi"]& &A_1\cap B_2 \ar[rr,blue,"\phi"]& &B_3 
\end{tikzcd}
\]
where the black lines represent containment of subgroups and are labelled by the corresponding indices $a_i, b_i$.
It is easy to see that $a_{i}\ge a_{i+1}$ and $ b_{i}\ge  b_{i+1}$, and so these sequences stabilize. 
We also see that $a_i b_{i+1} =  b_i a_{i+1}$ and therefore if $ b_{i}>a_{i}$ then so is $ b_{i+1}>a_{i+1}$.
Let us denote by 
\begin{equation*}
    \bar a_n:=[G:A_n]= a_1\cdot \ldots \cdot  a_n\text{ and }\bar b_n:=[G:B_n]= b_1\cdot \ldots \cdot  b_n.
\end{equation*}

It follows that if $ b_1> a_1$ then 
\[
    \frac{\bar b_n}{\bar a_n}= \frac{ b_1\cdot\ldots \cdot  b_n}{ a_1\cdot\ldots \cdot a_n} \longrightarrow \infty.
\]

We summarize this discussion in the following lemma.

\begin{lemma}\label{lem: arbitrary indices}
If $G$ contains two isomorphic subgroups of different finite indices, then for every $\lambda>0$ there exist isomorphic subgroups $A,B\le G$ such that $\frac{[G:B]}{[G:A]}>\lambda$.
\end{lemma}

\begin{proof}[Proof of Theorem \ref{thm: isomorphic finite index}]
Let $G$ be a non-elementary hyperbolic group with globally stable cylinders.
By Theorem \ref{thm: free splitting}, we may assume that $G$ is one-ended. 
Let $\alpha,\beta>0$ be the constants from Theorem \ref{thm: main}. 
By Lemma \ref{lem: arbitrary indices}, if $G$ contains two isomorphic subgroups of different finite indices then if contains two subgroups $A,B$ such that $$\frac{[G:B]}{[G:A]}>\frac{\beta}{\alpha}.$$
On the other hand, by Theorem \ref{thm: main upgraded}, we have $$ \alpha [G:B] \le \Cx_{2,m}(B) = \Cx_{2,m}(A) \le \beta [G:A].$$ A contradiction.
\end{proof}

\bibliographystyle{plain}
\bibliography{biblio}
\end{document}